\documentclass{amsart}
\usepackage{amssymb}
\usepackage{amsmath}
\usepackage{amsthm}
\newcommand{\pn}{\par\noindent}
\newcommand{\pmn}{\par\medskip\noindent}
\newtheorem{theor}{Theorem}[section]
\theoremstyle{definition} 
\newtheorem{ex}{Example}[section] \theoremstyle{remark}
\newtheorem{rem}{Remark}[section]
\begin{document}
\title{On the construction of frieze patterns from partitions
of convex polygons by nonintersecting diagonals}
\author{Yury Kochetkov}
\date{}
\email{yukochetkov@hse.ru,yuyukochetkov@gmail.com}

\begin{abstract} We demonstrate in an elementary way how to construct a
frieze pattern of width $m-3$ from a partition of a convex $m$-gon
by not intersecting diagonals.
\end{abstract}

\maketitle

\section{Introduction}
\pn Frieze patterns were introduced by Coxeter in the work
\cite{Cox}: a frieze pattern is an array of finite many infinite
rows of numbers (infinite both to the right and to the left). Two
upper rows are row of zeroes and row of units and two bottom rows
are row of units and row of zeroes. Each next row is shifted with
respect to the previous one in the following way:
\[{\small\begin{tabular}{cccccccccccc} $\cdots$&0&&0&&0&&0&&0&&$\cdots$\\
$\cdots$&&1&&1&&1&&1&&1&$\cdots$\\
$\cdots$&a&&b&&c&&d&&e&&$\cdots$ \\ \end{tabular}}\] We demand
that for each four adjacent elements
$${\small\begin{array}{ccc}&a&\\b&&c\\&d&\end{array}}\text{ the relation }
\,\,bc-ad=1 \eqno(1)$$ is satisfied. The number of rows between
rows of units is called the \emph{width} of a frieze pattern. The
upper row of units has number $0$, and the next row has number
one. A sequence of $m+2$ elements, one in each row, beginning from
the upper row of units, where next element is shifted one step to
the right, is called \emph{diagonal}.

\begin{ex} In the table below a fragment of a frieze pattern of width 3 is
presented with a diagonal (underlined numbers):
\[{\scriptsize \begin{tabular}{ccccccccccccccc} 0&&0&&0&&0&&0&&0&&0&&0\\
&\underline{1}&&1&&1&&1&&1&&1&&1&\\ 1&&\underline{3}&&2&&1&&3&&2&&1&&3\\
&2&&\underline{5}&&1&&2&&5&&1&&2&\\ 1&&3&&\underline{2}&&1&&3&&2&&1&&3\\
 &1&&1&&\underline{1}&&1&&1&&1&&1&\\ 0&&0&&0&&0&&0&&0&&0&&0\\
\end{tabular}}\]  \end{ex}
\pn We will use the following properties of friezes \cite{Cox},
\cite{MG}.
\begin{itemize}
\item Rows of a frieze pattern of width $m$ are periodic and the
period divides $m+3$. \item Let $\ldots,a_1,a_2,\ldots$ be
elements of the first row, and let
$d=\{v_0=1,v_1=a_1,v_2,v_3,\ldots\}$ be the diagonal with the
first element $a_1$. Then $v_2=a_2\cdot v_1-1, v_3=a_3\cdot
v_2-v_1, \ldots,v_{i+1}=a_{i+1}\cdot
v_i-v_{i-1},\ldots,1=a_{m+1}\cdot v_m-v_{m-1},0=a_{m+2}\cdot
1-v_m$. Opposite is also true: a pattern of width $m$ with the
above property for all diagonals is a frieze pattern.
\end{itemize}

\pn Next statement is quite important.

\begin{theor} Numbers in the first row of a frieze pattern cannot
be all $\geqslant 2$. \end{theor}
\begin{proof} Indeed, let $F$ be a frieze pattern of width $m$ and let
$\ldots,a_1,a_2,\ldots$ be the first row. If $D$ is the diagonal
with the first element $d_1=a_1$, then $d_2=a_2d_1-1>d_1$,
$d_3=a_3d_2-d_1>d_2$, and so on. Thus, $d_{m+1}>1$, contradiction.
\end{proof}

\pmn Conway and Coxeter \cite{CC} demonstrated how to construct a
frieze pattern $F$ of width $m-3$ with positive integer elements
from a partition of a convex $m$-gon into triangles by not
intersecting diagonals. Given a partition we define the weight
$w(v)$ of a vertex $v$: $w(v)$ is the number of triangles
adjoining to $v$. Let $\{v_1,v_2,\ldots,v_m\}$ be the numeration
of vertices in the counterclockwise order. Then the periodic
sequence $\ldots,w(v_1),w(v_2),\ldots,w(v_m),\ldots$ is the first
row of $F$. The opposite statement is also true: each integer
frieze of width $m-3$ is generated from some $m$-gon in the above
way.

\begin{ex}\quad
\parbox{2cm}{\begin{picture}(50,40) \multiput(10,0)(0,40){2}{\line(1,0){30}}
\multiput(0,20)(40,20){2}{\line(1,-2){10}}
\multiput(0,20)(10,-20){2}{\line(2,1){40}}
\multiput(0,20)(40,-20){2}{\line(1,2){10}}
\put(0,20){\line(1,0){50}}
\end{picture}} \quad $\Rightarrow$\quad\quad
\parbox{7cm}{\scriptsize\begin{tabular}{ccccccccccccccccc}
$\ldots$&0&&0&&0&&0&&0&&0&&0&&0&$\ldots$\\
$\ldots$&&1&&1&&1&&1&&1&&1&&1&&$\ldots$\\
$\ldots$&2&&1&&3&&2&&1&&3&&2&&1&$\ldots$\\
$\ldots$&&1&&2&&5&&1&&2&&5&&1&&$\ldots$\\
$\ldots$&2&&1&&3&&2&&1&&3&&2&&1&$\ldots$\\
$\ldots$&&1&&1&&1&&1&&1&&1&&1&&$\ldots$\\
$\ldots$&0&&0&&0&&0&&0&&0&&0&&0&$\ldots$\\ \end{tabular}}
\end{ex}

\pn In work \cite{HJ} the Conway and Coxeter's construction was
generalized on an arbitrary partition of a convex polygon into
polygonal parts by non intersecting diagonals. It was done in
following way: for each $n$ we define the weight $w_n$ (positive
number) of $n$-gon (the weight of triangle is $1$). The weight of
a vertex is the sum of weights of parts, adjoining to it. Authors
of the work \cite{HJ} set $w_n=2\cos\left(\frac \pi n\right)$ and
proved that the Conway-Coxeter construction with these weights
indeed generates a frieze pattern of width $m-3$ from a general
partition of a convex $m$-gon.

\begin{ex} Let us consider a partition of pentagon
\parbox{1cm}{\begin{picture}(24,19)
\multiput(2,2)(0,10){2}{\line(1,0){20}}
\multiput(2,2)(20,0){2}{\line(0,1){10}} \put(2,12){\line(2,1){10}}
\put(12,17){\line(2,-1){10}} \end{picture}}  into triangle and
quadrangle. Here $w_3=1$ and $w_4=\sqrt 2$ and the generated
frieze pattern is presented below.
\[{\scriptsize \begin{tabular}{ccccccccccccccc}
$\cdots$&0&&0&&0&&0&&0&&0&&0&$\cdots$\\
$\cdots$&&1&&1&&1&&1&&1&&1&&$\cdots$\\ $\cdots$&$\sqrt 2$&&$\sqrt
2$&&$1+\sqrt 2$&&$1$&&$1+\sqrt 2$&&$\sqrt 2$&&$\sqrt 2$&$\cdots$\\
$\cdots$&&$1$&&$1+\sqrt 2$&&$\sqrt 2$&&$\sqrt 2$&&$1+\sqrt
2$&&$1$&&$\cdots$\\ $\cdots$&1&&1&&1&&1&&1&&1&&1&$\cdots$\\
$\cdots$&&0&&0&&0&&0&&0&&0&&$\cdots$\\ \end{tabular}}\]
\end{ex}

\pn Our aim is to recover weights $w_n=2\cos\left(\frac \pi
n\right)$ in an elementary way (the work \cite{HJ} is quite
difficult).

\section{General partitions. Weights}
\pn In this section we will prove that the number
$2\cos\left(\frac \pi n\right)$ indeed is the best possible choice
for the weight $w_n$. \pmn Let $t$ be the weight of $m$-gon $P$
and let $\sigma$ be the trivial partition of $P$ into one part.
Then all elements of the first row of the corresponding frieze
pattern will be $t$, all elements of the second row will be
$t^2-1$, all elements of the third row will be $t^3-2t$, and so
on. Let us define the sequence of polynomials $\{Q_1=x,
Q_2=x^2-1,\ldots,Q_n=x\cdot Q_{n-1}-Q_{n-2},\ldots\}$.
\begin{theor} $$Q_n=\sum_{0\leqslant k\leqslant n/2} (-1)^k
\binom{n-k}{k}\cdot x^{n-2k}.$$ \end{theor}
\begin{proof} Induction. \end{proof}
\pn If the trivial partition of $m$-gon and the weight $t$ indeed
generate a frieze, then $Q_{m-2}(t)=1$ and $Q_{m-1}(t)=0$. As the
weight is $<2$, then we can make the change of variable:
$x:=2\cos(\alpha)$. We have
$$\begin{array}{l} Q_1=2\cos(\alpha)\\ Q_2=4\cos^2(\alpha)-1=
2\cos(2\alpha)+1\\ Q_3=4\cos(\alpha)\cos(2\alpha)=2\cos(3\alpha)+
2\cos(\alpha)\\ Q_4=4\cos(\alpha)\cos(3\alpha)+4\cos^2(\alpha)-
2\cos(2\alpha)-1=2\cos(4\alpha)+2\cos(2\alpha)+1\\
\hspace{3cm} \cdots\quad\cdots\quad\cdots\\
Q_{2k}=2\cos(2k\alpha)+2\cos((2k-2)\,\alpha)+\ldots+2\cos(2\alpha)+1\\
Q_{2k+1}=2\cos((2k+1)\,\alpha)+2\cos((2k-1)\,\alpha)+\ldots+
2\cos(3\alpha)+2\cos(\alpha)\\ \hspace{3cm}\ldots\quad\ldots\quad
\ldots \end{array}$$

\begin{theor} If $\alpha=\frac \pi n$, then $Q_{n-2}(\alpha)=1$
and $Q_{n-1}(\alpha)=0$.\end{theor} \begin{proof} Let $n$ be odd,
$n=2k+1$, then
$$Q_{2k-1}(\alpha)=\sum_{i=1}^k 2\cos\left(\frac{(2i-1)\pi}{n}\right)=
\sum_{i=1}^{2k+1} \cos\left(\frac{(2i-1)\pi}{n}\right)-
\cos\left(\frac{(2k+1)\pi}{n}\right)=1.$$ Also,
$$Q_{2k}(\alpha)=1+\sum_{i=1}^k 2\cos\left(\frac{2i\pi}{n}\right)=
\sum_{i=1}^n\cos\left(\frac{2i\pi}{n}\right)=0.$$ Let $n$ be even,
$n=2k$, then
$$Q_{2k-2}(\alpha)=1+\sum_{i=1}^{k-1} 2\cos\left(\frac{2i\pi}{n}\right)=1,$$ because
$\cos\left(\frac{2i\pi}{n}\right)=-\cos\left(\frac{(n-2i)\pi}{n}\right)$
for $i\leqslant \frac k2$. And
$$Q_{2k-1}(\alpha)=\sum_{i=1}^k 2\cos\left(\frac{(2i-1)\pi}{n}\right)=0$$ for the
same reasons. \end{proof}

\begin{rem} Thus, polynomials $Q_{n-2}-1$ and $Q_{n-1}$ have a common factor.
Actually, common roots of these polynomials are numbers
$2\cos\left(\frac{i\pi}{n}\right)$ for $i$ odd and coprime with
$n$. However, weights $2\cos\left(\frac{k\pi}{n}\right)$,
$k\geqslant 3$ generate frieze patterns with negative entries.
\end{rem}

\section{General partitions. Frieze patterns}
\pn Now we can prove the theorem.

\begin{theor} A partition of an $m$-gon into polygonal parts by
non intersecting diagonals generates a frieze pattern of width
$m-3$. \end{theor}

\begin{proof} Induction. Let the statement is true for $m<n$ and
let $P$ be a convex $n$-gon and $\sigma$ be its partition into
polygons $P_1,\ldots,P_k$ with number of vertices
$i_1,\ldots,i_k$, respectively.  As
$i_1-2+i_2-2+\ldots+i_k-2=n-2$, then $i_j-1$ edges of some $P_j$
are edges of $P$. Thus, we can cut $P_j$ from $P$ and obtain the
$(n-i_j+2)$-gon $V$ with partition $\tau$. Set $s:=n-i_j+2$. The
partition $\tau$ generates a frieze pattern $F$ of width $s-3$
with the first row $\ldots,a_1,a_2,\ldots,a_s,\ldots$. The
$(s-2)$-nd element of every diagonal in $F$ is $1$ and the next
element
--- $0$. Let $r=i_j$ and let $t$ be the weight of $P_j$, $t=2\cos\left(\frac \pi
r\right)$. Thus, $Q_{r-2}(t)=1$ and $Q_{r-1}(t)=0$. We consider
the pattern $F'$ with the first row
$$\ldots,a_{i-1},a_i+t,\underbrace{t,\ldots,t}_{r-2},a_{i+1}+t,a_{i+2},\ldots$$
and will define further rows, using the diagonal properties. We
must prove that the $(n-2)$-nd (i.e. $(r+s-4)$-th) element of
every diagonal in $F'$ is $1$ and the next element --- $0$. \pmn
In what follows we will write $q_m$ instead of $Q_m(t)$.\pmn
\underline{The first case.} Let us construct the diagonal $d'$ in
$F'$ with elements $u_1=a_1,u_2,\ldots$ and let us consider the
auxiliary  diagonal $d$ in $F$ with elements $v_1=a_1,v_2,\ldots$.
We have: $u_1=v_1,\ldots,u_{i-1}=v_{i-1}$. Then,
\begin{multline*} u_i=v_i+v_{i-1}t,\,\,
u_{i+1}=v_it+v_{i-1}(t^2-1)=v_iq_1+v_{i-1}q_2,\ldots,
u_{i+r-3}=\\=v_iq_{r-3}+v_{i-1}q_{r-2}=v_iq_{r-3}+v_{i-1},\,\,
u_{i+r-2}=v_iq_{r-2}+v_{i-1}q_{r-1}=v_i.\end{multline*} Then
$$u_{i+r-1}=a_{i+1}v_iq_{r-2}+v_iq_{r-2}t-v_iq_{r-3}-v_{i-1}q_{r-2}=v_{i+1},\,\,
u_{i+r}=v_{i+2},\ldots$$ Thus, $d'$ differs from $d$ by insertion
of $r-2$ elements: $u_i,u_{i+1},\ldots,u_{i+r-3}$, i.e. the
diagonal $d'$ has $s+r-5$ elements, as demanded. \pmn
\underline{The second case.} Now let the first row be
$\ldots,\underbrace{\underline{t},t\ldots,t}_k,a_1+t,a_2,\ldots,a_{s-1},a_s+t,t,\ldots$
and the first element $f_1$ of the diagonal $d'$ be the underlined
$t$. We will consider two diagonals in $F$ with the fist elements
$v_1=a_1$ and $u_1=a_2$, respectively. We have:
$f_1=t,\,\,f_2=q_2,\ldots,f_k=q_k$. Then
\begin{multline*} f_{k+1}=a_1q_k+q_kt-q_{k-1}=v_1q_k+q_{k+1},\,\,f_{k+2}=v_2q_k+u_1q_{k+1},
\ldots,\\f_{k+s-2}=a_{s-2}v_{s-3}q_k+a_{s-2}u_{s-4}q_{k+1}-v_{s-4}q_k-u_{s-5}q_{k+1}
=v_{s-2}q_k+u_{s-3}q_{k+1}=q_k+u_{s-3}q_{k+1}.\end{multline*} Then
$$
f_{k+s-1}=v_{s-1}q_k+u_{s-2}q_{k+1}=u_{s-2}q_{k+1}=q_{k+1},\,\,
f_{k+s}=u_{s-1}q_{k+1}+tq_{k+1}-q_k=q_{k+2},\ldots$$ Thus, the
diagonal $d'$ is the result of the insertion of $s-2$ elements
$f_{k+1},\ldots,f_{k+s-2}$ into the sequence $q_1,\ldots,q_{r-3}$,
i.e. $d'$ has $s+r-5$ elements, as demanded. \pmn Special cases
$k=0, k=r-2, f_1=a_1+t$, can be considered in the same way.
\end{proof}

\begin{ex} Let us consider the partition of an octagon into
triangle, quadrangle and pentagon:
\begin{center}{\parbox{20mm}{\begin{picture}(40,40)
\multiput(10,0)(0,40){2}{\line(1,0){20}}
\multiput(0,10)(40,0){2}{\line(0,1){20}}
\multiput(0,10)(30,30){2}{\line(1,-1){10}}
\multiput(0,30)(30,-30){2}{\line(1,1){10}}
\put(0,10){\line(2,1){40}}
\put(10,0){\line(1,1){30}}\end{picture}}
\parbox{8cm}{The weight of triangle is $1$, of quadrangle --- $\sqrt 2$, of
pentagon --- $t$, where $t^2=t+1$.}}\end{center} This partition
generates the frieze pattern of width 5 with 8-periodic rows (here
we write "s" instead of $\sqrt 2$):
\[{\scriptsize \begin{tabular}{cccccccccccccccc}
 0&&0&&0&&0&&0&&0&&0&&0&\\
&1&&1&&1&&1&&1&&1&&1&&1\\
$1+t$&&$1+s$&&$s$&&$s$&&$1+s+t$&&$t$ &&$t$&&$t$&\\
&$s+t+st$&&$1+s$&&$1$&&$1+s+st$
&&$2t+st$&&$t$&&$t$&&$2t$\\
$t+2st$&&$1+t+st$&&$1$&&$1+t$&&$t+2st$&&$1+t+st$&&$1$&&$1+t$&\\
&$2t+st$&&$t$&&$t$&& $2t$&&$s+t+st$&&$1+s$&&$1$&&$1+s+st$\\
$1+s+t$&&$t$&&$t$&&$t$&&$1+t$&&$1+s$&&$s$&&$s$&\\
&1&&1&&1&&1&&1&&1&&1&&1\\ 0&&0&&0&&0&&0&&0&&0&&0&\\
\end{tabular}}\]
\end{ex}

\vspace{5mm}
\end{document}